\documentclass[12pt]{article}

\usepackage{fullpage,times,url,natbib}

\usepackage{amsthm,amsfonts,amsmath,amssymb,epsfig,color,float,graphicx,verbatim}
\usepackage{algorithm,algorithmic}

\usepackage{hyperref}
\hypersetup{
	colorlinks   = true, 
	urlcolor     = blue, 
	linkcolor    = blue, 
	citecolor   = black 
}

\newtheorem{theorem}{Theorem}

\newtheorem{lemma}{Lemma}

\newcommand{\reals}{\mathbb{R}}

\newcommand{\Lcal}{\mathcal{L}}

\newcommand{\norm}[1]{\|#1\|}

\renewcommand{\eqref}[1]{Eq.~(\ref{#1})}

\newcommand{\thmref}[1]{Thm.~\ref{#1}}

\newcommand{\Prob}{\mathbb{P}}
\newcommand{\cS}{\mathcal{S}}

\newcommand{\Span}{\operatorname{span}}

\title{A Simple Complexity Lower Bound for Solving $Ax=b$}
\author{Ohad Shamir\\University of Toronto and Weizmann Institute of Science}
\date{}

\begin{document}
\maketitle

\begin{abstract}
In this note, we provide a short and direct proof that approximately solving $Ax=b$ to relative error $\varepsilon$, where $A$ has condition number $\kappa$ and unrestricted dimension, requires $\Omega(\kappa\log(1/\varepsilon))$ matrix-vector multiplications in the worst case, even for randomized algorithms. This essentially recovers the lower bound of \cite{DerezinskiEpperlyMeyer2026} for this setting, whose elegant and more general approach inspired us to seek a short direct proof. A straightforward reduction implies the classical $\Omega(\sqrt{\kappa}\log(1/\varepsilon))$ lower bound for optimizing strongly convex quadratic functions, applicable to randomized algorithms. 
\end{abstract}

\section{Introduction}

Consider the well-studied problem of solving a system of linear equations, or equivalently solving
\begin{equation}\label{eq:axb}
\min_{x\in \reals^n} \|Ax-b\|^2~,
\end{equation}
for some nonsingular $n\times n$ matrix $A$ with condition number $\kappa$, and a vector $b$. In a recent landmark paper, Derezi{\'n}ski, Epperly and Meyer \citep{DerezinskiEpperlyMeyer2026} provided an essentially optimal $\Omega(\kappa\log(1/\varepsilon))$ lower bound on the number of matrix-vector products $x\mapsto (Ax,A^\top x)$ required to solve such problems in the worst case, up to relative error $\epsilon$ (matching up to constants, for example, the complexity of the conjugate gradient method applied to the normal equations). Importantly, the lower bound applies even for randomized algorithms, whereas the analysis of deterministic algorithms (via resisting oracles) is much simpler \citep{nemirovskiyudin1983,nesterov2018lectures}. 

To prove their result, the authors utilized very elegant and sophisticated techniques, potentially applicable to many other problems as well (involving trace inverse estimation, properties of Wishart matrices, total variation and data-processing inequalities, duality and more). This machinery naturally makes the proof rather involved, and reliant on intermediate results imported from other papers, in particular \cite{chewi2024query}. For such a fundamental computational problem, we believe it is useful to also have a relatively simple proof of its complexity lower bound.

The goal of this note is to provide such an analysis: It is short and mostly self-contained, involving some rather standard facts from approximation theory, elementary probabilistic arguments, and one concentration of measure result. We make no claim that the proof is especially novel, and we indeed borrow many of the ideas appearing in  \cite{DerezinskiEpperlyMeyer2026} (and previous literature discussed therein), such as utilizing polynomial inapproximability results, and random orthogonal matrices encoding rotations that cannot be revealed by a limited number of matrix-vector products. However, the details differ, allowing for a relatively compact and direct proof.

\textbf{AI usage.} This note is based on conversations with ChatGPT 5.6 Sol, aimed at understanding the analysis in \cite{DerezinskiEpperlyMeyer2026} and exploring how a similar result can be obtained in a shorter and more self-contained manner. Both the author and the AI contributed ideas, proof strategies, and proof drafts. We hope this illustrates a use of AI aimed at improving human understanding of an existing result, rather than replacing it. The author thoroughly revised and checked the final write-up, and takes full responsibility for it. 

\section{Setup and Main Result}

We consider an oracle model, where an algorithm observes $b$, and has access to the matrix $A$ via matrix-vector products. Specifically, at each iteration $t=1,2,\ldots,T$, the algorithm can  choose a vector $x_t$ (based on the information it obtained so far) and receive $(Ax_t,A^\top x_t)$. The algorithm is not computationally limited, and is possibly randomized. We assume that after $T$ such iterations, the algorithm returns $x_T$ (this is without loss of generality, since a separate output vector can always be considered as another query, increasing $T$ by at most $1$). In this model, we ask how many such iterations are required in the worst case to minimize $\|Ax-b\|^2$ to some given accuracy, over all matrices $A$ with condition number at most $\kappa$, and where the dimension $n$ is unrestricted. 

Since we focus on a lower bound, we lose nothing by  restricting $A$ to be a symmetric matrix. Such a matrix can always be decomposed as 
\[
A=UDU^{\mathsf T}
\]
for a diagonal matrix $D$ and orthogonal matrix $U$. Note that in this case, $A=A^\top$, so multiplying a vector by $A$ or its transpose is equivalent.

Under this setup, we prove the following:
\begin{theorem}\label{thm:main}
	The following holds for some universal constants $c,C>0$. For any $\kappa\ge C$ and $0<\varepsilon \le c$, any (possibly randomized)
algorithm as above, which satisfies
\[
\Prob_{\text{Alg}}\!\left(\norm{Ax_T-b}\le\varepsilon\norm b\right)\ge\frac23
\]
for every $b$ and $A$ with condition number $\leq \kappa$ and associated $U$, must make at least
\[
  T~\geq~c\,\kappa\log(1/\varepsilon)
\]
oracle queries. The result already holds for
symmetric matrices $A=UDU^\top$ with spectrum in
$[-\kappa,-1]\cup[1,\kappa]$, even with $D$ revealed to the algorithm in advance, and even if the algorithm uses the more powerful oracle
\begin{equation}\label{eq:oracle}
	x\mapsto \left(Ax,U^\top x\right).
\end{equation}
The lower bound holds already in dimension $n=\Omega(T^2)$. 
\end{theorem}

\textbf{Reduction to strongly convex quadratic optimization.} The theorem readily implies a lower bound of $\Omega(\sqrt{\kappa}\log(1/\epsilon))$ on the number of gradient/value evaluations required for solving strongly-convex quadratic problems, even with randomized algorithms (here, $\kappa$ is the condition number of the Hessian matrix). This reduction is folklore and straightforward: Indeed, consider the class of functions
$$ f(x)=\frac12\|Ax-b\|^2, $$
for symmetric $A$, so that 
$$
\nabla f(x)=A^2x-Ab, \qquad \nabla^2 f=A^2\succ 0, \qquad \operatorname{cond}(A^2)=\operatorname{cond}(A)^2. 
$$
A first-order (that is, gradient and value) oracle for such functions can be simulated using two products with $A$: Given a point $x$ compute \(y=Ax\), then \(A(y-b)\) is the gradient at $x$, and \(\frac12\|y-b\|^2\) is the function value. 

Since $f^*:=\inf_x f(x)=0$ for nonsingular $A$, the optimization criterion $ f(x)-f^\star\le \delta\bigl(f(0)-f^\star\bigr)$ translates to $ \|Ax-b\|\le\sqrt{\delta}\,\|b\|$. Applying \thmref{thm:main} with \(\varepsilon=\sqrt\delta\) implies a lower bound of
$$ \Omega\!\left( \sqrt{\operatorname{cond}(\nabla^2 f)} \,\log(1/\delta) \right) $$
on the number of required first-order oracle queries.

\section{Proof of \thmref{thm:main}}

The proof is based on a randomized strategy where $U$ is a uniformly random orthogonal matrix (thus, encoding a random rotation on $\reals^n$), and $D$ is chosen in some fixed manner, so as to encode some inapproximability results for low-degree polynomials. We then show that no deterministic algorithm (with the more powerful oracle in \eqref{eq:oracle}) can succeed against this randomized strategy, which by Yao's minimax principle implies a lower bound for randomized algorithms. 

\subsection{Step 1: Standard Polynomial Inapproximability Results}

Define
\[
  \cS_\kappa:=[-\kappa,-1]\cup[1,\kappa]~~~,~~~
  \rho_d(\kappa)
  :=\inf_{\substack{\deg p\le d\\p(0)=0}}
    \max_{\lambda\in\cS_\kappa}|1-p(\lambda)|.
\]

The following is a straightforward inapproximability result for low-degree polynomials:
\begin{lemma}\label{lem:chebyshev}
For every $\kappa\ge2$ and $d\ge1$,
\[
  \rho_d(\kappa)\ge \exp\!\left(-\frac{2d}{\kappa}\right)~.
\]
\end{lemma}

\begin{proof}
Since $\cS_\kappa$ is symmetric around the origin, an optimal $p$ may be taken even, so $p(x)=s(x^2)$ with
$\deg s\le m:=\lfloor d/2\rfloor$ and $s(0)=0$.  Setting $r(y):=1-s(y)$, we have $\deg r\le m$ and $r(0)=1$. Recall the following
standard Chebyshev extremal property: if a degree-$m$ polynomial $q$
satisfies $|q(z)|\le1$ for all $z\in[-1,1]$, then for every $|z|>1$,
\[
|q(z)|\le |T_m(z)|,
\]
where $T_m$ is the degree-$m$ Chebyshev polynomial (see \citet[\S2.7, Eq.~(2.37)]{Rivlin1990}). Applying this
after mapping $[1,\kappa^2]$ affinely to $[-1,1]$, with $0$ mapped to
$-(\kappa^2+1)/(\kappa^2-1)$, leads to
\[
\max_{y\in[1,\kappa^2]}|r(y)|
\ge
\frac{1}{
	T_m\!\left(\frac{\kappa^2+1}{\kappa^2-1}\right)}.
\]
Writing
$(\kappa^2+1)/(\kappa^2-1)=\cosh\theta$ gives
$\theta=\log((\kappa+1)/(\kappa-1))\le4/\kappa$.  Since
$T_m(\cosh\theta)=\cosh(m\theta)\le e^{m\theta}$ and $m\le d/2$, the
claim follows.  
\end{proof}

The following lemma implies a similar inapproximability result, even if we only care about a weighted average error over a certain set of $\leq d+1$ points (rather than the maximal error):

\begin{lemma}\label{lem:witness}
For every $d\ge1$, there is some $r\leq d+1$, distinct points
$\lambda_1,\ldots,\lambda_r\in\cS_\kappa$ and positive weights
$w_1,\ldots,w_r$ summing to $1$, such that
\begin{equation}\label{eq:witness}
  \inf_{\substack{\deg p\le d\\p(0)=0}}~\sum_{s=1}^r w_s |1-p(\lambda_s)|^2
  = \rho_d(\kappa)^2.
\end{equation}
\end{lemma}

\begin{proof}
	The $\leq$ direction is immediate, since a maximum is larger than a weighted average. As to the $\geq$ direction, let $\mathcal V_d=\{p:\deg p\le d,\ p(0)=0\}$, and let $p_*$ be a best
	uniform approximant (over $\mathcal V_d$) to the constant function $1$ on $\cS_\kappa$. By a well-known result in approximation theory \cite[Theorem~2.3]{KrooPinkus2010}, there are
	$r\le d+1$ distinct points
	$\lambda_s\in A_{1-p_*}$ (the set of points where $|1-p_*|$ attains its maximal value), positive weights $w_s$ (which we normalize so that $\sum_s w_s=1$), and signs $\sigma_s=\operatorname{sign}(1-p_*(\lambda_s))$ such that
	\[
	\sum_{s} w_s\sigma_s q(\lambda_s)=0
	\qquad(q\in\mathcal V_d).
	\]
	Moreover, by definition of $\lambda_s$ and $\sigma_s$,
	\[
	1-p_*(\lambda_s)=\sigma_s\rho_d(\kappa).
	\]
	
	Now fix any $p\in\mathcal V_d$. Since $p-p_*\in\mathcal V_d$, the preceding
	two displayed equations imply
	\[
	\sum_s w_s\sigma_s(1-p(\lambda_s))
	=
	\sum_s w_s\sigma_s(1-p_*(\lambda_s))
	=
	\rho_d(\kappa).
	\]
	Applying Cauchy--Schwarz (with respect to the inner-product space $\left\langle x,y\right\rangle=\sum_{s} w_s x_s y_s$) on the left hand side, and noting that $\sum_s w_s\sigma_s^2=1$, it follows that
	\[
	\rho_d(\kappa)
	\le
	\left(\sum_s w_s|1-p(\lambda_s)|^2\right)^{1/2}~.
	\]
	Squaring and taking the infimum
	over $p\in\mathcal V_d$ proves the $\geq$ direction of \eqref{eq:witness}.
\end{proof}

\subsection{Step 2: A Random Orthogonal $U$ Hides Information}

The following is an elementary fact about random rotations on $\reals^n$ (see for example \cite[Chapter~1]{Meckes2019}). Intuitively, it says that knowing how some rotation acts on some subspace of $\reals^n$ gives us no information about its action on the orthogonal subspace:
\begin{lemma}\label{lem:haar-completion}
Let $U \in \reals^{n \times n}$ be a uniformly random orthogonal matrix (drawn from the Haar measure on $O(n)$), and let $\Lcal \subseteq \reals^n$ be a fixed linear subspace. Conditioned on the restriction $U|_{\Lcal}$, the mapping $U|_{\Lcal^\perp} \colon \Lcal^\perp \to (U\Lcal)^\perp$ is distributed as a uniformly random orthogonal map.
\end{lemma}

To see why this is helpful, define
\[
v_t := U^\top x_t,
\]
and note that with our oracle (as defined in \eqref{eq:oracle}), a query $x_t$ returns
\[
  U^{\mathsf T}x_t=v_t~~~\text{and}~~~
  Ax_t=UDv_t~.
\]
Together with the algorithm's knowledge of $x_t=Uv_t$ (the image of $v_t$ under $U$), we see that the oracle responses up to iteration $t$ reveals information on how $U$ acts on the subspace spanned by $v_1,Dv_1,\ldots,v_t,Dv_t$. However, no  information is obtained about how $U$ acts on the orthogonal subspace. Thus, by Lemma \ref{lem:haar-completion}, conditioned on the past it remains a uniformly random orthogonal map on the orthogonal subspace. As a result, the $v_t$ vectors can be shown to inhabit a linear subspace that evolves rather ``noisily''. This is formalized in the following key lemma:
\begin{lemma}\label{lem:fresh}
Fix an $n\times n$ matrix $D$, let $U$ be a uniformly random $n\times n$ orthogonal matrix, $b$ an independent standard Gaussian random vector on $\reals^n$, and let $g_0=U^\top b$. Fix a deterministic algorithm which given $(D,b)$, makes $T$ adaptive queries to the oracle defined in \eqref{eq:oracle}. Then there exist random vectors $g_1,\ldots,g_T$ (possibly dependent on $U$) such that 
\begin{enumerate}
	\item \label{item:Dv} Almost surely, $Dv_t
		\in
		\Span\{D^j g_i:0\le i\le t,\ 1\le j\le t\}$ for every $t\le T$.
	\item \label{item:gg} The joint distribution of $g_0,\ldots,g_T$ is that of independent standard Gaussian vectors in $\reals^n$.
\end{enumerate}
\end{lemma}

\begin{proof}
	For \(t\ge0\), define
	\[
	\Lcal_t
	:=
	\Span\{g_0,v_1,Dv_1,\ldots,v_t,Dv_t\},
	\qquad
	\text{where}~~v_t:=U^{\mathsf T}x_t.
	\]
	We also let $\mathcal F_t$ denote the history through query $t$ (consisting of $b$ and the oracle queries and responses so far), together with
	$g_0,\ldots,g_t$. We construct \(g_1,\ldots,g_T\) recursively and prove by induction that after
	query \(t\):
	
	\begin{enumerate}
		\item \(g_0,\ldots,g_t\) are independent \(N(0,I_n)\) vectors
		\item For every \(s\le t\), $Dv_s
		\in
		\Span\{D^j g_i:0\le i\le s,\ 1\le j\le s\}
		$
		\item
		$
		\Lcal_t
		\subseteq
		\Span\{D^j g_i:0\le i\le t,\ 0\le j\le t\},
		$
		and conditional on $\mathcal F_t$,
		\(U|_{\Lcal_t^\perp}\) is a uniformly random orthogonal map onto
		\((U\Lcal_t)^\perp\).
	\end{enumerate}
	
	The first two assertions at \(t=T\) are exactly the lemma statement. The third assertion is an auxiliary statement to help prove the induction. 
	
	Starting with the base case \(t=0\), the first assertion holds since $g_0=U^\top b$, and a rotation of a standard Gaussian is still standard Gaussian. The second assertion is vacuous. As to the third assertion, the first part holds since $\Lcal_0=\Span\{g_0\}$. As to the second part, note that $g_0$ is independent of $U$ (since
	conditioned on any fixed $U$, $U^\top b\sim N(0,I_n)$ with a law that does
	not depend on $U$). Thus, conditioning on $g_0$ does not affect the distribution of $U$, while
	conditioning additionally on $b=Ug_0$ fixes the restriction of $U$ to
	$\Lcal_0=\Span\{g_0\}$. Lemma~\ref{lem:haar-completion} therefore
	gives the uniform-orthogonal-map part of the third assertion.
		
	Now assume the assertions hold through query \(t-1\), and condition on $\mathcal F_{t-1}$. Then \(\Lcal_{t-1}\), the restriction
	\(U|_{\Lcal_{t-1}}\), and the next query \(x_t\) are fixed.  Decompose
	\[
	x_t=x_\parallel+x_\perp~~~\text{where}~~~
	x_\parallel\in U\Lcal_{t-1},
	\quad
	x_{\perp}\perp U\Lcal_{t-1}.
	\]
	Note that \(U^{\mathsf T}x_\parallel\in \Lcal_{t-1}\), and $U^\top x_{\perp} \perp \Lcal_{t-1}$ (since for any $r\in \Lcal_{t-1}$, $(U^\top x_{\perp})^\top r = x_{\perp}^\top Ur=0$ by definition of $x_{\perp}$). 
	
	We now describe how $g_t$ is constructed:
	\begin{itemize}
		\item If $\|x_\perp\|\neq0$, define
		\[
		\xi_t:=\frac{U^\top x_\perp}{\|x_\perp\|}.
		\]
		By the induction hypothesis,
		$\xi_t\mid\mathcal F_{t-1}$ is uniform on the unit sphere of
		$\Lcal_{t-1}^\perp$. Independently sample
		\[
		a_t\sim N(0,\Pi_{\Lcal_{t-1}}),
		\qquad
		R_t^2\sim\chi^2_{n-\dim\Lcal_{t-1}},
		\]
		(that is, a uniform standard Gaussian on $\Lcal_{t-1}$, and a Chi-square random variable with $n-\dim\Lcal_{t-1}$ degrees of freedom), and set $g_t=a_t+R_t\xi_t$. This implies that $g_t$ has a standard Gaussian distribution on $\reals^n$ (via the standard polar decomposition of a Gaussian vector, e.g. 
		\cite[Section~3.3]{Vershynin2018}). 		
		\item 	If $\|x_\perp\|=0$, simply take
		$g_t\sim N(0,I_n)$ independently of $\mathcal F_{t-1}$.
	\end{itemize} 
	Either way, the law of $g_t|\mathcal F_{t-1}$ is standard Gaussian. This proves the first induction assertion at time $t$. Moreover, by construction of $g_t$, we have that $v_t=U^\top x_t = U^\top x_{\parallel}+U^\top x_{\perp}$ is a sum of a vector in $\Lcal_{t-1}$ and a vector proportional to the projection $\Pi_{\Lcal_{t-1}^\perp} g_t$, which implies
	\begin{equation}\label{eq:vtk}
		v_t=U^{\mathsf T}x_t
		\in
		\Lcal_{t-1}+\Span\{g_t\}~.
	\end{equation}
	
	The second induction assertion holds by applying $D$ to 
	Eq. \ref{eq:vtk} and using the induction hypothesis to get
	\begin{align}
		Dv_t
		&\in
		D\Lcal_{t-1}+\Span\{Dg_t\} \notag\\
		&\subseteq
		\Span\{D^j g_i:0\le i\le t-1,\ 1\le j\le t\}
		+\Span\{Dg_t\}\notag\\
		&\subseteq
		\Span\{D^j g_i:0\le i\le t,\ 1\le j\le t\}\label{eq:vtkinc}~.
	\end{align}
	
	It remains to prove the third induction assertion. The part about $\Lcal_t$ follows from 
	\[
	\Lcal_t
	\stackrel{(*)}{=}
	\Lcal_{t-1}+\Span\{v_t,Dv_t\}~\stackrel{(**)}\subseteq~
	\Lcal_{t-1}+\Span\{g_t,Dv_t\}
	 \stackrel{(***)}{\subseteq}
	\Span\{D^j g_i:0\le i\le t,\ 0\le j\le t\}~,
	\]
	where $(*)$ is by definition of $\Lcal_t$, $(**)$ is by Eq. \eqref{eq:vtk}, and $(***)$ is by induction hypothesis and \eqref{eq:vtkinc}. 
	
	It remains to verify the uniform-orthogonal-map part of the third assertion. We wish to argue that conditioned on $\mathcal F_{t-1}$, the oracle responses at round $t$ and the construction of $g_t$ do not leak information about $U|_{\Lcal_t^\perp}$, which by Lemma \ref{lem:haar-completion} implies the desired assertion. To see this, note that the oracle responses at round $t$ reveal
$
Uv_t=x_t$ and
$
U(Dv_t)=Ax_t,
$
which only pertain to $U|_{\Lcal_t}$. As to $g_t$, note that by construction, it possibly depends on $U|_{\Lcal_t^\perp}$ only when $x_\perp\neq0$, in which case the dependence is through $\xi_t=\frac{U^\top x_{\perp}}{\|x_{\perp}\|}.
$
Moreover,
\[
\xi_t
=
\frac{v_t-U^\top x_\parallel}{\|x_\perp\|}
\in\Lcal_t,
\qquad
U\xi_t=\frac{x_\perp}{\|x_\perp\|}.
\]

Thus, the information about $U$ revealed by $g_t$ also pertains only to $U|_{\Lcal_t}$. More formally, conditioned on $\mathcal F_{t-1}$, we may first condition on $\xi_t$; by Lemma~\ref{lem:haar-completion}, the restriction of $U$ to $(\Lcal_{t-1}+\Span\{\xi_t\})^\perp$ remains a uniformly random orthogonal map. Conditioning next on $U(Dv_t)=Ax_t$ and applying Lemma~\ref{lem:haar-completion} once more shows that the restriction to $\Lcal_t^\perp$ remains uniformly random (the remaining randomness used to construct $g_t$ is independent of $U$). Thus, conditioned on $\mathcal F_t$, 
\[ U|_{\Lcal_t^\perp} : \Lcal_t^\perp\to (U\Lcal_t)^\perp \]
is a uniformly random orthogonal map.
\end{proof}

\subsection{Step 3: Construction and Proof of Theorem \ref{thm:main}}

As discussed earlier, by Yao's minimax principle, it suffices to construct, for each query budget $T$, a hard distribution over $(A,b)$ pairs on which every deterministic algorithm fails to achieve some error level with constant probability.

First, by Lemma \ref{lem:witness}, there are distinct values $\lambda_1,\ldots,\lambda_r$ in $[-\kappa,-1]\cup [1,\kappa]$ (for some $r\leq T+1$) and corresponding weights $w_1,\ldots,w_r$ summing to $1$, such that for any degree-$T$ polynomial $p$ with $p(0)=0$, it holds that
\begin{equation}\label{eq:cheblow}
\sum_{s=1}^{r}w_s(1-p(\lambda_s))^2~\geq \rho_T(\kappa)^2.
\end{equation}
For some sufficiently large constant $C$, define
\[
N:=C^2T^2,
\qquad
m_s:=\lceil Nw_s\rceil+\lceil CT\rceil ,
\]
and let $D$ be an $n\times n$ diagonal matrix consisting of  $\lambda_1,\ldots,\lambda_r$ with multiplicities $m_1,\ldots,m_r$ (so that $n=\sum_s m_s$). Note that for large enough $C$,
\begin{equation}\label{eq:size}
  n=O(T^2),
  \qquad
  N\ge0.99n,
  \qquad
  m_s\ge CT,
\end{equation}
and by \eqref{eq:cheblow}, every polynomial $p$ of degree at most $T$ with $p(0)=0$ satisfies
\begin{equation}\label{eq:multiplicity}
  \sum_s m_s(1-p(\lambda_s))^2
  \ge N\rho_T(\kappa)^2.
\end{equation}

The hard distribution we consider is $A=UDU^\top$ with $D$ as above and $U$ a uniformly random orthogonal matrix, and $b\sim N(0,I_n)$ independently. Recall that the algorithm is assumed to be provided with $D,b,$ and $T$ queries of the oracle specified in \eqref{eq:oracle}. 

The algorithm runs for $T$ iterations, and returns $x_T$. By  Lemma~\ref{lem:fresh}, there are vectors  $g_0,\ldots,g_T$ distributed as independent standard Gaussians, such that $U^\top b=g_0$ and 
$
  DU^{\mathsf T} x_T
  =Dv_T
  \in
  \Span\{D^j g_i:0\le i\le T,\ 1\le j\le T\}.
$
Therefore, there are polynomials $p_0,\ldots,p_T$ (whose coefficients generally depend on the random elements $U,b,g_0,\ldots,g_T$) satisfying
\begin{equation}\label{eq:poly}
  DU^{\mathsf T}x_T
  =\sum_{i=0}^T p_i(D)g_i,
  \qquad
  \deg p_i\le T,
  \qquad
  p_i(0)=0~.
\end{equation}
Note that $p_i(0)=0$ is due to $D^j g_i$ ranging only over $j\geq 1$. 
Now, let $\mathcal E_s$ be the eigenspace of $D$ corresponding to $\lambda_s$, and define the projections
\[
  g_{i,s}:=\Pi_{\mathcal E_s}g_i,
  \qquad
  \mathcal H_s:=\Span\{g_{1,s},\ldots,g_{T,s}\},
  \qquad
  h_s:=\Pi_{\mathcal H_s^\perp}g_{0,s}.
\]
Since $g_0,\ldots,g_T$ are independent standard Gaussians, $\|h_s\|^2$ (conditioned on $\mathcal{H}_s$) has a chi-square distribution with $m_s-T$ degrees of freedom. By standard
chi-square concentration (e.g.\ \cite[Lemma~1]{LaurentMassart2000}) and a
union bound, for sufficiently large $C$, it holds with probability $\geq 0.9$ that
\begin{equation}\label{eq:good}
  \norm{h_s}^2\ge\frac34m_s\quad\text{for every }s~~~\text{and}~~~
  \norm{b}^2\le\frac54n.
\end{equation}

We now have all the ingredients to conclude the proof. By \eqref{eq:poly}, the residual $Ax_T-b$ of the vector $x_T$ returned by the algorithm satisfies
\[
  U^{\mathsf T}(Ax_T-b)
  =DU^{\mathsf T}x_T-g_0 = (p_0(D)-1)g_0+\sum_{i=1}^{T} p_i(D)g_i~.
\]
Therefore, decomposing across the different $\mathcal E_s$ subspaces,
\[
\norm{Ax_T-b}^2~=~ \norm{U^\top (Ax_T-b)}^2~=~\sum_s \left\|(p_0(\lambda_s)-1)g_{0,s}+\sum_{i=1}^{T}p_i(\lambda_s)g_{i,s}\right\|^2~.
\]
Projecting the vector in each norm on $\mathcal H_s^{\perp}$ can only reduce norms, so the above is at least
\[
\sum_s \left\|(p_0(\lambda_s)-1)h_s\right\|^2~=~ \sum_{s} \|h_s\|^2 (1-p_0(\lambda_s))^2~.
\]
Conditioned on the event in \eqref{eq:good} (which holds with probability at least $0.9$), and using \eqref{eq:multiplicity}, this is at least 
\[
\frac34\sum_s m_s(1-p_0(\lambda_s))^2~
\ge~ \frac34N\rho_T(\kappa)^2~.
\]
(note that this is valid, since \eqref{eq:multiplicity} holds deterministically for every realization of the possibly-random $p_0$). Overall, we get that with probability at least $0.9$, $\norm{Ax_T-b}^2\geq \frac{3}{4} N\rho_T(\kappa)^2$. Moreover, under the same event, $\norm{b}^2\leq \frac{5}{4}n$ by \eqref{eq:good}, and thus
\[
\frac{\norm{Ax_T -b}}{\norm b}\geq \sqrt{\frac{3N/4}{5n/4}}\cdot \rho_T(\kappa)~\geq~
\sqrt{\frac{3\cdot 0.99 n/4}{5n/4}}\cdot \rho_T(\kappa)\geq \frac{3}{4} \rho_T(\kappa)~,
\]
where we used the assumption that $N\geq 0.99n$ from \eqref{eq:size}. By Lemma~\ref{lem:chebyshev}, the above is at least $\frac{3}{4}\exp(-2T/\kappa)$. This is at least some prescribed $\epsilon$ (assuming $\epsilon$ small enough) for any $T\leq \frac{\kappa}{3}\log(1/\epsilon)$. 

To conclude, we get that with the same randomized construction, any deterministic algorithm with $T\leq \frac{\kappa}{3}\log(1/\epsilon)$ fails to return an $\epsilon$-optimal solution with some constant probability (all this assuming $\epsilon$ small enough and $n=O(T^2)$ sufficiently large). Theorem~\ref{thm:main} now follows by Yao's minimax principle. 

\textbf{Acknowledgements.} We thank Raphael Meyer for several very helpful comments.

\bibliographystyle{plainnat}
\bibliography{bib}

\end{document}